\documentclass[10pt,twoside]{article}
\usepackage{amssymb}

\usepackage{amsmath}
\usepackage{latexsym}

\newcommand{\COM}[1]{}
\renewcommand{\theequation}{\arabic{section}.\arabic{equation}}

\newtheorem{theorem}{Theorem}[section]

\newtheorem{lemma}{Lemma}[section]
\newtheorem{proposition}{Proposition}[section]
\newtheorem{corollary}{Corollary}[section]

\newtheorem{remark}{\normalfont\scshape Remark}[section]

\newenvironment{proof}{\noindent\textsc{Proof.\/}}{}
\newenvironment{pf}[1]{\noindent\textsc{Proof of #1.\/}}{}

\newcommand{\subj}[2]{\textsf{AMS 2000 subject classifications.}
Primary {#1}; Secondary {#2}.\newline}
\newcommand{\key}[1]{\textsf{Keywords and phrases.} {#1}.\newline}
\newcommand{\abb}[1]{\textsf{Abbreviated title.} {#1}.}

\newcommand{\fot}[5]{\renewcommand\thefootnote{}
\footnotetext{\parindent=0.0mm \vskip-3mm \subj{#1}{#2}\key{#3}\abb{#4}
\newline\textsf{Date.} \date{\today}}}

\def\vsb{\hfill$\Box$}

\def\vsp{\vskip-8mm\hfill$\Box$\vskip3mm}

\newcommand{\veps}{\varepsilon}

\newcommand{\be}{\begin{equation}}
\newcommand{\ee}{\end{equation}}
\newcommand{\bea}{\begin{eqnarray}}
\newcommand{\eea}{\end{eqnarray}}
\newcommand{\beaa}{\begin{eqnarray*}}
\newcommand{\eeaa}{\end{eqnarray*}}
\newcommand{\beal}{\begin{aligned}}
\newcommand{\eeal}{\end{aligned}}

\newcommand{\var}{\mathrm{Var\,}}
\newcommand{\cov}{\mathrm{Cov\,}}

\newcommand{\sumk}{\sum^n_{k=1}}

\newcommand{\ttt}[1]{\quad\mbox{ #1}\quad}

\newcommand{\asto}{\stackrel{a.s.}{\to}}
\newcommand{\pto}{\stackrel{p}{\to}}

\newcommand{\dto}{\stackrel{d}{\to}}

\newcommand{\nifi}{n\to\infty}

\newcommand{\erw}{elephant random walk}
\newcommand{\cfn}{{\cal F}_n}
\newcommand{\cfnm}{{\cal F}_{n-1}}
\newcommand{\cgn}{{\cal G}_n}

\newcommand{\mem}{{\mathfrak M}}
\newcommand{\xnp}{X_{n+1}}\newcommand{\xnm}{X_{n-1}}
\newcommand{\tnp}{T_{n+1}}\newcommand{\tnm}{T_{n-1}}
\newcommand{\tnz}{T_n^2}
\newcommand{\tnpz}{T_{n+1}^2}
\newcommand{\exnpf}{E(X_{n+1}\mid \cfn)}
\newcommand{\etnpf}{E(T_{n+1}\mid {\cal F}_n)}

\begin{document}
\date{}
\title{\textsf{Variations of the elephant random walk}}
\author{Allan Gut\\Uppsala University \and Ulrich Stadtm\"uller\\
Ulm University}
\maketitle

\begin{abstract}\noindent
In the classical simple random walk the steps are independent, viz., the walker has no memory. In contrast, in the \erw, which was introduced by Sch\"utz and Trimper \cite{erwdef} in 2004,  the walker remembers the whole past, and the next step always depends on the whole path so far. Our main aim is  to prove analogous results when the elephant has only a restricted memory, for example remembering only the most remote step(s), the most recent step(s) or both. We also extend the models to cover more general step sizes.
\end{abstract}

\fot{60F05, 60G50,}{60F15, 60J10}{Elephant random walk, law of large numbers, asymptotic (non)normality, method of moments, difference equation, Markov chain}
{Elephant random walk}

\section{Introduction}
\setcounter{equation}{0}
\markboth{A.\ Gut and U.\ Stadtm\"uller}{Elephant random walk}

In the classical \emph{simple\/} random walk the steps are equal to plus or minus one and independent---$P(X=1)=1-P(X=-1)=p$, ($0<p<1$). In this model the walker has no memory. This random walk is, in particular, Markovian. Motivated by applications, although interesting in its own right, is the case when the walker has some memory. The extreme case is, of course, when the walker has a complete memory, that is, when ''the next step'' depends on the whole process so far. This so called \erw\ (ERW) was introduced by Sch\"utz and Trimper \cite{erwdef} in 2004,  the name being inspired by the fact that elephants have a very long memory.

The first, more substantial, paper on \erw s is, to the best of our knowledge, Bercu's paper \cite{bercu}, in which he proves a number of limit theorems. A main point is that there is a kind of phase transition at the point $p=P(X=1)=3/4$, which divides the problem into the diffusive regime, $0\leq p<3/4$, the critical regime, $p=3/4$, and the superdiffusive regime, $3/4<p\leq1$, with somewhat different asymptotics.

A main device in his paper is the use of martingale theory due to the observation that a multiplicative scaling of the random walk constitutes a martingale.

Our main interest is the situation in which the elephant has only a limited memory, either that he or she remembers only some distant past, only a recent past or a mixture of both. No paper on exact results and proofs seems to exist, only simulations, in which case a given fraction of the distant/recent past is remembered; \cite{scv, csv,moura}.  

The first task in this direction is to consider the cases when the walker only remembers the first (two) step(s) or only the most recent (previous) step. In particular the latter case involves rather cumbersome computations and we therefore invite the reader(s) to try to push our results further. It should also be mentioned that the paper by Engl\"ander and Volkov \cite{volkov} is devoted to this latter case, although from a different angle, in that the next step is not generated by flipping a coin, rather by turning it over or not. They have a somewhat different focus, in particular, they consider the case with different $p$-values in each step. 

The cases with limited memory behave very differently mathematically in that some of the walks are still non-markovian others are markovian, but there is no convenient martingale around. Moreover there are no phase transitions in these cases.

A second point concerns  the extension of (some of) Bercu's results in \cite{bercu} from the simple random walk to general sums, that is,  to the case when the steps have an arbitrary distribution on the integers.

We begin by defining the various models in Section \ref{defs}. After some preliminaries in Section \ref{aux},  some results  for general ERW:s are obtained in Section \ref{allgemein}.  Sections \ref{anfang1} and \ref{anfang2} are devoted to the distant  past and Sections \ref{ende1} and \ref{ende2} to the recent past, respectively. These ''one-sided'' memories are then followed up in Sections \ref{beidesa} and \ref{beidesb} where we consider mixed cases, that is, when the memory contains some early steps as well as some recent ones, after which we shortly discuss some different models. We close with a section containing some  questions and remarks. For easier reading we collect some of the somewhat more lengthy (elementary and tedious) computations in the Appendix.

\section{Background}
\setcounter{equation}{0}\label{defs}
The \erw\ is defined as a simple random walk, where, however, the steps are not i.i.d.\ but dependent as follows. The first step $X_1$ equals 1 with probability $r\in [0,1]$ and is equal to $-1$ with probability $1-r$. After $n$ steps, that is, at position $S_n=\sumk X_k$, one defines
\[\xnp=\begin{cases} +X_{K},\ttt{with probability} p\in[0,1], \\-X_{K},\ttt{with probability} 1-p,\end{cases}\]
where $K$ has a uniform distribution on the integers $1,2,\ldots,n$. With $\cgn=\sigma\{X_1,X_2,\ldots,X_n\}$ this means (formula (2.2) of \cite{bercu}) that  
\bea
E(\xnp\mid\cgn)=(2p-1)\cdot\frac{S_n}{n},\label{21}\eea
 after which, setting  $a_n=\Gamma(n)\cdot\Gamma(2p)/\Gamma(n+2p-1)$, it turns out that
$\{M_n=a_nS_n,\,n\geq 1\}$ is a martingale.

\smallskip
Our main aim is  to extend these results to the case when the elephant has only a restricted memory, for example remembering only the most remote step(s) and/or the most recent one(s). A result in Section \ref{allgemein} allows us to conclude that our results remain true (suitably modified) also when the steps of the ERW:s follow a general distribution on the integers.

First in line is the case when the elephant only remembers the distant past, the most extreme one being when the memory is reduced to the first step only, viz.,
\[\xnp=\begin{cases} +X_{1},\ttt{with probability} p\in[0,1], \\-X_{1},\ttt{with probability} 1-p.\end{cases}\]
Somewhat more sophisticated is when the memory covers the first two steps, in which case
\[\xnp=\begin{cases} +X_{K},\ttt{with probability} p\in[0,1], \\-X_{K},\ttt{with probability} 1-p,\end{cases}\]
where $P(K=1)=P(K=2)=1/2$.
\smallskip

Technically more complicted is when the elephant only remembers the recent past. Here we focus on the very recent past, which is the last step, that is,
\[\xnp=\begin{cases} +X_{n},\ttt{with probability} p\in[0,1], \\-X_{n},\ttt{with probability} 1-p.\end{cases}\]

\smallskip
We begin, throughout,  by assuming that $X_1=1$, and generalize our findings in this setting (for simplicity) to the case  $r=p$. We denote our partial sums with $T_n$, $n\geq1$, when the first variable(s) is/are fixed and let $S_n$ be reserved for the case when they are random.
\smallskip

In order to move from $T_n$ to $S_n$ we also need to discuss the behavior of the walk when the initial value equals $-1$. However, in that case the evolution of the walk is the same \emph{except\/} for the fact that the trend of the walk is reversed, viz., the corresponding walk equals the mirrored image in the time axis. This implies that the mean after $n$ steps equals $-E(T_n)$, but the dynamics being the same, implies that the variance remains the same ($\var( -Y)=\var( Y)$ for a random variable $Y$). In fact, the second moments of the walk remain the same. The same goes for higher order moments---odd moments equal the negative of those when $X_1=1$, and even moments remain the same. In Sections \ref{anfang2} and \ref{beidesb} we depart from the assumption that $X_1$ and $X_2$ are fixed, and then the additional case $X_1+X_2=0$ has to be taken care of.

\smallskip
Finally, in order to avoid special effects we assume throughout that $0<p<1$; note that $p=1$ corresponds to $X_n=X_1$ for all $n$, and $p=0$ the the case of alternating summands.

\section{Some auxiliary material}
\setcounter{equation}{0}\label{aux}

For easier access of the arguments below we shortly present some auxiliary results from probability and analysis.

 \subsection{Disturbed limit distributions}\label{crsl}

The following (well-known) result (which is a special case of the Cram\'er--Slutsky theorem) will be used in order to go from a special case to a more general one.

\begin{proposition}\label{knep} Let $\{U_n,\,n\geq1\}$ be a sequence of random variables and suppose that $V$ is independent of all of them. If $U_n\dto U$ as $\nifi$, then $U_nV\dto UV$ as $\nifi$.
\end{proposition}
\begin{proof} Using characteristic functions and bounded convergence we have, as $\nifi$,
\beaa
\varphi_{U_nV}(t)&=&E\exp\{itU_nV\}=E\Big(E(\exp\{itU_nV\}\mid V)\Big)=E\varphi_{U_n}(tv)\to E\varphi_{U}(tv)\\
&=&E\Big(E(\exp\{itUV\}\mid V)\Big)=E\exp\{itUV\})=\varphi_{UV}(t).
\eeaa
An application of the continuity theorem for characteristic functions finishes the proof.\vsb\end{proof}

\subsection{Conditioning in case of a restricted memory}
\label{xnk}
Let $\{S_n,\,n\geq1\}$ be an ERW, let $\{\cfn,\,n\geq1\}$ denote the   $\sigma$-algebras generated by the memory of the elephant and let $\cgn= \sigma\{X_1,X_2,\ldots,X_n\}$ stand for  the full memory. We already know from (\ref{21}) above that $E(\xnp\mid\cgn)=(2p-1)S_n/n$. Our aim is to establish analogs when the elephant has a restricted memory, that is, analogs for $E(\xnp\mid\cfn)$. 

Toward the end, let $I_n=\{i\leq n: i \in \mem\}$, where $\mem =$ the memory of the elephant. Then,
\bea\label{uli}
E(\xnp\mid \cfn)=p\cdot \sum_{i\in I_n}\frac1{|I_n|}X_i + (1-p)\cdot\sum_{i\in I_n}\frac1{|I_n|}(-X_i)
=(2p-1)\cdot\frac{\sum_{i\in I_n}X_i}{|I_n|},
\eea
that is, the conditional mean equals the average of the possible choices multiplied by the expected value of the sign;  in analogy with (\ref{21}).

If, for example, $I_n=\{n\}$  the elephant only remembers the most recent step, and $I_n=\{1,n\}$ means that he/she only remembers the first and the most recent steps; these are two cases that will be considered in the sequel. In these cases (\ref{uli}) states that
\[E(\xnp\mid \cfn)=(2p-1)X_n\ttt{and} E(\xnp\mid \cfn)=(2p-1)\frac{X_1+X_n}{2},\]
respectively. 

The next problem is when we condition on steps that are not contained in the memory. In words, if they do not, the elephant does not remember them, and, hence, cannot choose among them in a following step. More precisely, mathematically $\mem$ is defined as those steps in the past on which the elephant bases the next step. Technically, let $I\subset \{1,2,\ldots,n\}$ be an arbitrary set of indices, such that $I\cap I_n=\emptyset$. Then
\bea\label{uli2}
E(\xnp\mid \sigma\{I_n\cup I\})=E(\xnp\mid\cfn)=(2p-1)\frac{\sum_{i\in I_n}X_i}{|I_n|}.\eea
It follows, in particular, that
\bea\label{uli3}
E(\xnp\mid \cgn)=(2p-1)\frac{\sum_{i\in I_n}X_i}{|I_n|},\eea
and that
\bea\label{uli4}
E(S_n\xnp\mid \cgn)=
S_nE\big(\xnp\mid\cgn\big)
=S_n(2p-1)\frac{\sum_{i\in I_n}X_i}{|I_n|}.\eea
This, and the fact that $\xnp^2=1$, will be useful several times for the computation of second moments as follows:
\bea\label{uli5}
E(S_{n+1}^2)&=&E(S_n+\xnp)^2=E(S_n^2) +2E(S_n\xnp)+E(\xnp^2)\nonumber\\
&=&E(S_n^2)+\frac{2(2p-1)}{|I_n|}E\Big(S_n\sum_{i\in I_n}X_i\Big) +1.
\eea

\subsection{Difference equations}\label{pfdifferences}  
In the proofs we use several difference equations. For convenience and easy reference we summarize here some well-known facts about linear difference equations that are used on and off.
\begin{proposition}\label{diff}
\emph{(i)} Consider the first order equation
\[x_{n+1}=a\, x_n +b_n, \ttt{for} n \ge 1, \ttt{with} x^*_1 \ttt{given.}\]
Then
\[x_n=a^{n-1}x^*_1+\, \sum_{\nu=0}^{n-2}a^\nu b_{n-1-\nu}.\]
If, in addition, $|a|<1$ and $b_n=bn^\gamma$ with $\gamma>-1$, then
\[x_n=\frac{b_{n-1}}{1-a} -\frac{\gamma ab_{n-1}}{n(1-a)^2}\big(1+o(1)\big)\ttt{as}\nifi.\]
\emph{(ii)} If, in particular, $|a|<1$ and $x_{n+1}=ax_n+b$, then
\[x_n=\frac{b}{1-a}+a^{n-1}\big(x_1^*-\frac{b}{1-a}\big)=\frac{b}{1-a}\big(1+o(1)\big)\ttt{as}\nifi.\]
\emph{(iii)} Next is the homogeneous, second order equation
\[x_{n+1}=a\,x_n+b\,x_{n-1}, \ttt{for} n \ge 2,\ttt{with} x^*_1,\,x^*_2 \ttt{given.}\]
Then, with $\lambda_{1/2}= (a \pm \sqrt{a^2+4b})/2$, provided $a^2+4b\not=0$, 
\[x^h_n=c_1\lambda_1^n+c_2 \lambda_2^n \ttt{with} c_1,\,c_2 \ttt{chosen such that} x^h_i=x^*_i \ttt{for}i=1,2\,. \]
\emph{(iv)} As for the inhomogeneous second order equation
\[x_{n+1}=a\,x_n+b\,x_{n-1}+d_n, \ttt{for} n \ge 2, \ttt{with} x^*_1,\,x^*_2 \ttt{given,}\]
we have $x_n=x^h_n+y_n$, where $y_n$ is some solution of the inhomogeneous equation, where the constants $c_1,c_2$ in $x^h_n$ are chosen properly.
If $d_n\equiv d$ and $a+b\not=1$ we may choose $y_n=d/(1-a-b)$. 
\end{proposition}

\subsection{Some notation}\label{notation}
We use the standard $\delta_a(x)$ to  denote the distribution function with a jump of height one at $a$.
Constants $c$ and $C$ are always numerical constants that may change between appearances.

\section{General \erw s}\label{allgemein}
\setcounter{equation}{0}
Let $\{\widetilde{S}_n,\,n\geq1\}$ be an ERW, and suppose that $R$ is a random variable with distribution function $F_R$ that is independent of the walk. If $\widetilde{S}_n/a_n\asto Z$ as $\nifi$ for some normalizing positive sequence $a_n\to\infty$ as $\nifi$, and some random variable $Z$, it follows from Proposition \ref{knep} that $R\widetilde{S}_n/a_n\asto RZ$ as $\nifi$. An immediate consequence of this fact is that we can extend Theorems 3.1, 3.4 and (the first half of) Theorem 3.7 of \cite{bercu} to cover more general step sizes. Namely, consider the ERW for which $\widetilde{X}_1\equiv 1$, and let the random variables $\widetilde{X}_n$, $n\ge 2$, be constructeded as in Section 2 with this special $\widetilde{X}_1$ as starting point.   Furthermore, let $R$ be a random variable, independent of $\{ \widetilde{X}_n,\,n\geq1\}$,  and consider
$X_n=R\cdot \widetilde{X}_n$, $n\geq1$, and, hence, $S_n= R\cdot \widetilde{S}_n$. 

The following theorem (which reduces to the cited results of \cite{bercu} if $R$ is a coin-tossing random variable),  holds for $S_n=R\tilde{S}_n$:
\begin{theorem}\quad
\emph{(a)}\quad For $0<p<3/4$, \quad $\dfrac{S_n}{n}\asto 0\ttt{as}\nifi$;\\[1mm]
\emph{(b)}\quad For $p=3/4$, \quad $\dfrac{S_n}{\sqrt{n}\log n}\asto 0\ttt{as}\nifi$;\\[1mm]
\emph{(c)}\quad For $3/4<p<1$, \quad $\dfrac{S_n}{n^{2p-1}}\asto RL\ttt{as}\nifi$,\\[1mm]
where $L$ is a non-dgenerate random variable.
\end{theorem}
As for convergence in distribution,  we have to distinguish more carefully between the three cases.  
\begin{theorem}\label{thmgerw}
For $0<p<3/4$ we obtain
\[ \frac{S_n}{\sqrt{n}} \dto \int_{\mathbb{R}\backslash\{0\}}  {\cal N }_{0,\frac{1}{3-4p}} (\cdot/|t|) \,dF_R(t) +P(R=0)\cdot\delta_{[0,\infty)}(\cdot)\ttt{as}\nifi.\]
 Moreover, if $E(R^2)<\infty$, then $E(S_n/\sqrt{n})\to 0 $ and $E((S_n/\sqrt{n})^2)\to E(R^2)/(3-4p)$ as $\nifi$.
\end{theorem}
\begin{proof}
As $R$ and $S_n$ are independent we find that
\begin{eqnarray*} 
P\Big(\frac{S_n}{\sqrt{n}}\le x\Big)&=& \int_{\mathbb{R}} P\Big(R\frac{\widetilde{S}_n}{\sqrt{n}} \le x \mid R\Big) dF_R(t)\\
&=&  \int_{\mathbb{R}} P\Big(t\frac{\widetilde{S}_n}{\sqrt{n}} \le x \Big) dF_R(t)\\
&=& \int_{(-\infty,0)} P\Big(\frac{\widetilde{S}_n}{\sqrt{n}} \ge x/t \Big)  dF_R(t)+ \int_{(0,\infty)} P\Big(\frac{\widetilde{S}_n}{\sqrt{n}} \le x/t \Big) dF_R(t)\\&&\hskip3pc+ P(R=0) \cdot\delta_{[0,\infty)}(x)\\
&\to&  \int_{(-\infty,0)} \big(1-{\cal N }_{0,\frac{1}{3-4p}} ( x/t )\big) dF_R(t)+ \int_{(0,\infty)}  {\cal N }_{0,\frac{1}{3-4p}}( x/t ) \, dF_R(t)\\&&\hskip3pc+ P(R=0) \cdot\delta_{[0,\infty)}(x),
\end{eqnarray*}
by dominated convergence which yields the desired result.

The second part is immediate, since $R$ is independent of everything else.
\vsb
\end{proof}
\begin{remark}\label{remgerw}\emph{ If $R=\pm 1$ with probabilities $r$ and $(1-r)$, respectively, the limit distributions of $S_n/\sqrt{n}$ and $\widetilde{S}_n/\sqrt{n}$ are the same, and we rediscover Theorem 3.3 of \cite{bercu}.
}
\end{remark}
\begin{remark}\label{rem42}\emph{For the critical case, $p=3/4$ one similarly obtains, using \cite{bercu}, Theorem 3.6, that
\[ \frac{S_n}{\sqrt{n\log n}} \dto \int_{\mathbb{R}\backslash\{0\}}  {\cal N }_{0,1} (\cdot/|t|) \,dF_R(t) +P(R=0)\cdot\delta_{[0,\infty)}(\cdot)\ttt{as}\nifi.\]
The supercritical case, $3/4<p<1$, has a different evolution and no analogous result exists.}\vsb\end{remark}

\section{Remembering only the distant past 1}\label{anfang1}
\setcounter{equation}{0}
This turns out as being the easiest case, since convenient independence is inherent. We begin by assuming that the elephant only remembers the first step, i.e., that $\cfn=\sigma\{X_1\}$, and begin with the assumption that $X_1=1$ (recall that partial sums are denoted with the letter $T$). Then, 
\[\exnpf=E(\xnp\mid X_1)=(2p-1)\cdot 1=E(\xnp)\ttt{for all} n\geq 1,\] and, hence, 
\[\etnpf=1+ n(2p-1) = E(T_{n+1}).\]
Moreover, applying (\ref{uli5}) to $T_n$ we find that
\beaa
E(\tnpz)&=& E(T_n^2)+2\,(2p-1)E(T_n\cdot1)+1\\
&=&E(\tnz)+2(2p-1)\big(1+(n-1)(2p-1)\big)+1\\
&=&E(\tnz)+2(2p-1)^2n +4(2p-1)(1-p)+1,
\eeaa
which, after telescoping, yields
\[E(\tnpz)=1+(2p-1)^2n(n+1)+\big(4(2p-1)(1-p)+1\big)n,\]
and, finally, 
\[\var (\tnp) =\big(1-(2p-1)^2)\big)n=4p(1-p)n.\]
A completely analogous calculation for characteristic functions, with an eye on (\ref{uli2}), shows that
\beaa
\varphi_{\tnp}(t)&=&E\Big(E\Big(e^{it(T_n+\xnp)}\mid\cgn\Big)\Big)=E\Big(e^{itT_n}\cdot E\Big(e^{it\xnp}\mid \cgn\Big)\Big)\\
&=&E\Big(e^{itT_n}\cdot E\Big(e^{it\xnp}\mid X_1\Big)\Big)=E\Big(e^{itT_n}\cdot (pe^{it}+(1-p)e^{-it}\Big)\\
&=&\varphi_{T_n}(t)\cdot  \big(pe^{it}+(1-p)e^{-it}\big),
\eeaa
after which telescoping tells us that
\[\varphi_{\tnp}(t)=\big( pe^{it}+(1-p)e^{-it}\big)^n\cdot e^{it},\]
after which a standard computation shows that
\[\varphi_{(T_n-n(2p-1))/\sqrt{n}}(t)\to \exp\{-4p(1-p)t^2\}\ttt{as}\nifi.\]
Next we note that the computations so far prove that the increments are \emph{uncorrelated}, suggesting independence ... In fact, recalling that we have assumed that $X_1=1$, we have, setting $\alpha_k=1$ if $X_k=+X_1$, for  $k\geq2$, and 0 otherwise,
\beaa
P(X_i=X_j=1)&=&P(\alpha_i=\alpha_j=1-p)= P(\alpha_i=1-p)\cdot P(\alpha_j=1-p)\\
& =&(1-p)^2=P(X_i=1)\cdot P(X_j=1),
\eeaa
for $i,j\geq2$ and different.

This means that the ERW coincides with the classical simple random walk, except for the fact that the first step is always equal to one. This is---after some thinking---rather obvious, because (in the language of \cite{volkov}) we might interpret $X_1$ as a coin that we either flip or not before each new step. Hence we obtain:
\begin{proposition}\label{prop51}  The strong law of large numbers, the central limit theorem, and the law of the iterated logarithm all hold for  $\{T_n,\,n\geq1\}$.
\end{proposition}

 If, on the other hand, the first step is equal to $-1$, then, by symmetry,
$E(\tnp)=-n(2p-1)-1$, the variance remains the same (recall the discussion toward the end of Section \ref{defs}), and, again, by symmetry, $T_{n+1}+ n(2p-1)$ normalized by $\sqrt{n}$ is asymptotically normal and the SLLN and the LIL do hold again.

As a consequence, assuming that $X_1$ is a coin-tossing random variable, we are (asymptotically) confronted with two normal distributions, one for each of the two portions of the probability space. In fact, if we imagine the situation that 
$r=P(X_1=+1)$ is close to zero or one it  is rather apparent how the very first step determines along which branch it will evolve.  

One also notes, more formally, that  $E(\xnp\mid\cfn)=(2p-1)E(X_1)$, so that $E(\xnp)=(2p-1)^2$, implying that 
$\var( S_n)={\cal O}(n^2)$ (and not of order $n$) as $\nifi$. Thus, an ordinary CLT is not valid, with the exception that if $p=1/2$ the two ''branches'' determined by the first step collaps (asymptotically) into one, and we are ultimately faced with a classical simple symmetric random walk.

Hence, the following limit result is always available in the general case:
\begin{theorem}\label{thm52} Let $S_n=\sumk X_k$.Then, \\[2mm]
\emph{(a)} 
\hspace*{1cm}$\displaystyle\frac{S_n}{n}\dto\begin{cases}\phantom{-(}2p-1,&\ttt{with probability}p,\\[2mm]
 -(2p-1),&\ttt{with probability}1-p,\end{cases}\ttt{as}\nifi$;\\[2mm]
\emph{(b)} \hspace*{1cm} $E(S_n/n)\to (2p-1)^2\ttt{and} \var (S_n/n)\to  4p(1-p)(2p-1)^2\ttt{as} \nifi$.
\end{theorem}
\pf{(a)} If $X_1=\pm 1$ we know from above that $E(T_n)=\pm (1+(n-1)(2p-1))$, and that $\var(T_n)=4p(1-p)(n-1)$. This tells us that,  $\frac{T_n}{n}\pto \pm (2p-1)$ as $\nifi$. The conclusion follows.\\[1mm]
\pf{(b)}  Immediate.\vsb
\begin{remark}\label{remark51}\emph{(i) An interpretation of the limit in (a) is that the random walk at hand, on average, behaves, asymptotically, like a coin-tossing random variable with values at the points $\pm (2p-1)$.}\\[1mm]
\emph{(ii) An alternative way of phrasing the conclusion of the theorem is that}
\[F_{S_n/n}(x)\to p\cdot \delta_{-(2p-1)}(x)+(1-p)\cdot\delta_{2p-1}(x)\ttt{as}\nifi.\]
\end{remark}
However, if we use a random normalization we obtain the following result:
\begin{theorem}\label{thm53}  Let $S_n=\sumk X_k$.Then, \\[2mm]
\emph{(a)} \hspace*{1cm}$\displaystyle
\frac{S_n-n(2p-1)X_1}{\sqrt{4n p(1-p)}} \dto {\cal N}_{0,1}
 \ttt{as}\nifi$;\\[2mm]
\emph{(b)} \hspace*{1cm}$\displaystyle
\frac{S_n-n(2p-1)X_1}{n} \asto 0 \ttt{as}\nifi$;\\[2mm]
\emph{(c)} \hspace*{1cm}$\displaystyle
\limsup_{n \to \infty}\,(\liminf_{\nifi})\frac{S_n-n(2p-1)X_1}{\sqrt{8n p(1-p)\log \log n}} = 1\;(-1) \ttt{a.s.}$
\end{theorem}
\pf{(a)} We use the fact that
\[\frac{S_n-n(2p-1)X_1}{\sqrt{4n\,p\,(1-p)}}=X_1\, \frac{T_n-n(2p-1)}{\sqrt{4np(1-p)}},\]
together with Theorem \ref{thmgerw} and its Remark \ref{remgerw}.

Alternatively, one may condition on the value of $X_1$. This procedure will be exploited in the proof of Theorem \ref{thm62} in the next section. \\[1mm]
\pf{(b) and (c)} Define $\Omega_1=\{\omega \in \Omega : X_1(\omega)=1\}$ and $\Omega_2=\Omega_1^c$. After renormalization the original probability measure will be a probability measure on $\Omega_1$. Based on this measure on $\Omega_1$ we obtain an SLLN and an LIL for $S_n-n(2p-1)X_1$. Similarly on $\Omega_2$. Combining them yields the desired result.\vsb
 \begin{remark}\label{remark52}\emph{The strong law can also be formulated with a random RHS:
\[ \frac{S_n}{n}\asto (2p-1)X_1 \ttt{as} n \to \infty\,.\]}\end{remark}
\begin{remark}\label{remark53}\emph{If $X_1$ is a general random variable with distribution $F$ having no mass at zero, then}
\[ \frac{S_n-n(2p-1)\,X_1}{\sqrt{4np(1-p)}} \dto \int_{-\infty}^\infty {\cal N}_{0,1}(\cdot/|t|)\, dF(t)\ttt{as}\nifi.\] 
\vsp\end{remark}\vskip2mm
A special case is, once again, $p=1/2$:
\begin{corollary}\label{cor51}If $p=1/2$, then
\[\frac{S_n}{n}\asto 0 \mbox{ as }\nifi,\quad\frac{S_n}{\sqrt{n}}\dto  {\cal N}_{0,1}\mbox{ as }\nifi,\quad
\limsup_{n\to \infty}\,(\liminf_{\nifi})\frac{S_n}{\sqrt{2n\log\log n}}= 1;(-1) \ttt{a.s.}
\]
\end{corollary}

\section{Remembering only the distant past 2}\label{anfang2}
\setcounter{equation}{0}
In this section we begin by assuming that the elephant only remembers the first two steps, so that 
$\cfn=\sigma\{X_1, X_2\}$, and suppose that  $X_1=X_2=1$. Then, for $n\geq2$,
\bea\label{61}\exnpf=E(\xnp\mid X_1, X_2)=(2p-1)\cdot\frac{1+1}{2}=(2p-1)=E(\xnp)\eea 
for all $n\geq2$, and, hence,
\[\etnpf=2+ (n-1)(2p-1) =n(2p-1)+3-2p= E(T_{n+1}),\]
 (since $E(T_2)=2$).  Extending the idea from the previous section  that the walk evolves as an ordinary simple random walk beginning at the third step, a natural guess is that
\[\var (T_{n+1})=4p(1-p)(n-1),\quad n\geq3.\]
To see this we first observe that $\var( T_3)=\var( X_3)= 4p(1-p)$, that is, the formula is correct for $n=2$.
Assuming it is correct for $n-1$ we have 
\beaa
\var (T_{n+1})&=&\var(T_n+\xnp)=\var (T_n)+2\cov(T_n,\xnp)+\var( \xnp) \\
&=& 4p(1-p)(n-2)+ 0+ 4p(1-p)=4p(1-p)(n-1),\eeaa
since by (\ref{uli4}) and the fact that $X_1=X_2=1$,
\begin{equation}\label{cov}
\cov (T_n,\xnp)=E(T_n\xnp)-E(T_n)E(\xnp)=E(T_n(2p-1))-E(T_n)(2p-1)=0   .\end{equation}
Next, by modifying the computations involving the characteristic function from Section \ref{anfang1},
we obtain
\beaa
\varphi_{\tnp}(t)&=&E\Big(E\big(e^{it(T_n+\xnp)}\mid\cgn\big)\Big)=E\Big(e^{itT_n}\cdot E\big(e^{it\xnp}\mid X_1,X_2\big)\Big)\\
&=&E\Big(e^{itT_n}\cdot \big(pe^{it}+(1-p)e^{-it}\big)\Big)=\varphi_{T_n}(t)\cdot  \big(pe^{it}+(1-p)e^{-it}\big).
\eeaa
By continuing as before one obtains, after proper centering, a limiting normal distribution for these initial $X$-values. Similarly for the other ones. But, only for each ''branch'' separately. One can also ascertain that the variance is not linear if we assume random beginnings. Except, as before, when $p=1/2$ and the three main limit theorems (SLLN, CLT, LIL) hold (as in Corollary \ref{cor51}).
\smallskip

The following analog of Theorem \ref{thm52} holds in the general case (as one might expect):
\begin{theorem}\label{thm61} Let $S_n=\sumk X_k$. Then \\[2mm]
\emph{(a)} 
\hspace*{1cm}$\displaystyle\frac{S_n}{n}\dto\begin{cases}\phantom{-(}2p-1,&\ttt{with probability}p^2,\\
\phantom{5555}0,&\ttt{with probability}1-p,\\
 -(2p-1),&\ttt{with probability}p(1-p),\end{cases}\ttt{as}\nifi$;\\[2.5mm]
\emph{(b)} \hspace*{1cm} $E(S_n/n)\to p(2p-1)^2\ttt{and that} \var (S_n/n)\to  p(1-p)(2p-1)^2\big(4p^2+1\big).$
\end{theorem}
{\pf{(a)} 
If $X_1=X_2=\pm 1$ we know from above that $E(T_n)=\pm (n(2p-1)+3-2p)$, and that $\var( T_n)=4p(1-p)(n-2)$. Moreover, $E(T_n)=0$ whenever $X_1$ and $X_2$ have different signs. The variance remains the same (with $p=1/2$).  This, together with the fact that $P(X_1=X_2=1)=p^2$, $P(X_1=X_2=-1)=(1-p)p$, and $P(X_1\neq X_2)=p(1-p)+(1-p)^2=1-p$ helps us to finish the proof of the first part. Part (b) follows.\vsb
\begin{remark}\emph{(i) In analogy with Remark \ref{remark51} we have the interpretation  that the elephant, asymptotically, on average, performs a random walk on the points $\pm (2p-1)$ and $0$.}\\[2mm]
\emph{(ii) Mimicing Remark \ref{remark51} we may rewrite the conclusion of the theorem is
\[F_{S_n/n}(x)\to p(1-p)\cdot \delta_{-(2p-1)}(x)+(1-p)\cdot \delta_{0}(x) + p^2\cdot\delta_{2p-1}(x)\ttt{as}\nifi.\]
\vsp}\end{remark}
Once again random normalization produces further limit results:
\begin{theorem}\label{thm62}  Let $S_n=\sumk X_k$.Then, \\[2mm]
\emph{(a)} \hspace*{1cm}$\displaystyle
\frac{S_n-n(2p-1)\,(X_1+X_2)/2}{\sqrt{n }} \dto p\cdot{\cal N}_{0,4p(1-p)}+(1-p)\cdot {\cal N}_{0,1}
 \ttt{as}\nifi$;\\[2mm]
\emph{(b)} \hspace*{1cm}$\displaystyle
\frac{S_n-n(2p-1)\,(X_1+X_2)/2}{n} \asto 0 \ttt{as}\nifi$;\\[2mm]
\emph{(c)} \hspace*{1cm}$\displaystyle
\limsup_{n \to \infty}(\liminf_{n \to \infty})\frac{S_n-n(2p-1)\,(X_1+X_2)/2}{\sqrt{2n\log \log n}} = \sigma(X_1,X_2)\; (- \sigma(X_1,X_2))
 \ttt{a.s.,}$\\[1mm]
where $\sigma(X_1,X_2)= \begin{cases} 4p(1-p), &\ttt{for}  \omega \in \{\omega\in \Omega : X_1(\omega)\cdot X_2(\omega)=1\},\\ 1, &\ttt{otherwise.}\end{cases}$
\end{theorem}
\pf{(a)}  Conditioning on the value of $(X_1+X_2)/2$ we obtain
\beaa
&&\hskip-2pc P\Big(\frac{S_n-n(2p-1)(X_1+X_2)/2}{\sqrt{n}} \le x\Big) 
= P\Big(\frac{S_n-n(2p-1)(X_1+X_2)/2}{\sqrt{n}} \le x \mid X_1=X_2=1\Big)\cdot p^2\\
&&\hskip3pc +P\Big(\frac{S_n-n(2p-1)(X_1+X_2)/2}{\sqrt{n}} \le x \mid  X_1=X_2=-1\Big)\cdot p(1-p)\\
&&\hskip3pc+P\Big(\frac{S_n-n(2p-1)(X_1+X_2)/2}{\sqrt{n}} \le x \mid  X_1+ X_2=0\Big)\cdot (1-p)\\
&&\hskip1pc= P\Big(\frac{T_n-n(2p-1)}{\sqrt{n}} \le x\Big)\cdot p^2
+P\Big(\frac{-T_n+n(2p-1)}{\sqrt{n}} \le x\Big)\cdot p(1-p)\\
&&\hskip3pc+P\Big(\frac{T_n}{\sqrt{n}} \le x \mid  X_1+ X_2=0\Big)\cdot (1-p)\\
&&\hskip1pc \to \big(p^2 + p(1-p)\big)\cdot{\cal N}_{0,4p(1-p)}(x)+(1-p) \cdot{\cal N}_{0,1}(x)\\[2mm]
&&\hskip1pc=p\cdot{\cal N}_{0,4p(1-p)}(x)+(1-p) \cdot{\cal N}_{0,1}(x)\ttt{as}\nifi.
\eeaa
\noindent
Parts (b) and (c) follow along the lines of the proof of Theorem \ref{thm52}.\vsb

\section{The distant past; higher order}\label{anfang3}
\setcounter{equation}{0}

If one remembers the first $m$ random variables for some  $m \in \mathbf{N}$, the following obvious extension of the  above results emerges.
\begin{theorem}\label{general ess} For  $q_k=P(S_{m}=m-2k)$, $r_k=\big((m-k)p+k(1-p)\big)/m$, and $ p_k=(m-2k)(2p-1)/m$, where $ 0\le k \le m$ and $m \in\mathbf{N}$, 
\[\frac{S_n}{n} \dto \sum_{k=0}^{m} q_k \delta_{p_k}\ttt{as}\nifi,\] 
and
\[\frac{S_n-n(2p-1)S_{m}/m}{\sqrt{n}}\dto \sum_{k=0}^{m}q_k \,{\cal N}_{0,4r_k(1-r_k)} \quad \nifi.\]
\end{theorem}
\begin{proof}
As before we write
\begin{eqnarray*}
\lefteqn{\hspace*{-.9cm}P\Big(\frac{S_n}{n} \le x\Big)= P\Big(\frac{S_n}{n}\le x\mid X_1=1,\dots, X_m=1\Big) \cdot P( X_1=1,\dots, X_m=1)}\\
&& \hspace*{1cm}+\;P\Big(\frac{S_n}{n}\le x\,\big| \,\mbox{exactly one }\!\!-1 \mbox{ among the first  $m\; X$:es\Big)} \\
&& \hspace*{2cm} \times P(\mbox{exactly one }\!\! -1 \mbox{ among the first $m\; X$:es)}\\
&& \hspace*{1cm} +\cdots +P\Big(\frac{S_n}{n}\le x\mid X_1=-1,\dots, X_m=-1\Big) \cdot  P( X_1=-1,\dots, X_m=-1)\\
&&= P\Big(\frac{S_n}{n}\le x\mid X_1=1,\dots, X_m=1\Big) \cdot P( S_m=m)\\
&& \hspace*{1cm}+\;P\Big(\frac{S_n}{n}\le x\mid\mbox{exactly one }\!\! -1 \mbox{ among the first  $m\; X$:es\Big)} 
 \cdot P(S_m=m-2)\\
&&  \hspace*{1cm}+\cdots +P\Big(\frac{S_n}{n}\le x\mid \, X_1=-1,\dots, X_s=-1\Big) \cdot  P( S_m=-m),
\end{eqnarray*}
and observe that in each conditional case we have a random walk with the appropriate success probabilities, i.e., for  $S_m=m-2k$ the sucess probability is $r_k=\big((m-k)p+k(1-p)\big)/m$, and, hence, the expectation is  $p_k=2r_k-1=(m-2k)(2p-1)/m$.
\end{proof}
\begin{remark}\label{rem72}
\emph{(i) The probabilities at the jumps are relatively complicated and therefore not expressed in detail, but $q_0=p^m$ and $q_m=(1-p)p^{m-1}$.\\[1mm]
(ii) A more detailed analysis shows that the probability mass of the limit distribution of $S_n/n$ concentrates near zero  as $m $ increases.}\\[1mm]
\emph{(iii) One easily checks that the variance for each ''branch'' equals $4p(1-p)(n-m)$, which, in turn, is dominated by  $\leq 4p(1-p)n$, which, consequently, tells us that the analog of Theorems \ref{thm52} and \ref{thm61} holds.}\\[1mm]
\emph{(iv) Once again, the case $p=1/2$ is special as described in the two previous sections.}\vsb\end{remark}

\section{Remembering only the recent past 1}\label{ende1}
\setcounter{equation}{0}
This situation is much more complex, because, even though one remembers only recent steps, the path depends on the whole history so far (some remarks on that will be given in Subsection \ref{dep}). Once again we begin by assuming that the elephant only remembers the very last step, which means that $\cfn=\sigma\{X_n\}$. This setting is reminiscent of \cite{volkov}, where one turns over a coin instead of tossing it. The main focus there, however, is on different $p$-values at each step and, e.g., how this may affect phase transitions and behavior at critical values.

We begin, as always,  by assuming that 
$X_1=1$. Then, $E(X_1)=1$, and
\[\exnpf=E(\xnp\mid X_n)=(2p-1)\cdot X_n,\ttt{and, hence,} E(\xnp)=(2p-1)E(X_n),\]
for all $n\geq2$. By iterating this it follows that for, $n\ge 0$,
\bea\label{etnp}E(\xnp)=(2p-1)^{n}E(X_1)=(2p-1)^{n},\eea
and 
\[E (T_{n+1})=\frac{1-(2p-1)^{n+1}}{2(1-p)}.\] 
For the second moment we have, by (\ref{uli5}) and (\ref{uli2}),
\[ E(\tnpz\mid \cgn)=T_n^2 +2T_n(2p-1)X_n+1, \ttt{ hence, } E(\tnpz)=E(T_n^2) + 2(2p-1)E(T_nX_n)+1\,.\]
For the middle term we obtain by (\ref{uli2}), 
\[ E(T_nX_n)= E(X_n^2) + E(T_{n-1}E(X_n \mid {\cal G}_{n-1}))= 1+(2p-1)E(T_{n-1}X_{n-1})\,, \]
which in turn, after iteration, yields
\[E(T_nX_n)=1+\sum_{k=1}^{n-1}(2p-1)^k= \frac{1-(2p-1)^n}{2(1-p)}\,.\]
Now we can calculate the second moment:
\[E(\tnpz) = E(T_n^2) +2(2p-1) \cdot\frac{1-(2p-1)^n}{2(1-p)} +1
= E(T_n^2) +\frac{p}{1-p} -\frac{(2p-1)^{n+1}}{1-p}.\]
By telescoping we obtain
\[E(\tnpz)= \frac{np}{1-p} +{\cal O}(1)\ttt{as}\nifi,\]
which implies the following formula for the asymptotic variance:
\begin{equation}\label{vartnp}
\var(T_{n+1})= \frac{np}{1-p}+{\cal O}(1) \ttt{as} \nifi.
\end{equation}
Noticing that $S_n=X_1T_n$ and  that  $X_1=\pm1$,  a glance at (\ref{etnp}) and (\ref{vartnp}) shows that $\frac{T_n}{n}\pto0$ and  that
$\frac{S_n}{n}\pto0$ as $\nifi$, suggesting the following result:
\begin{theorem}\label{clttn} For $X_1=\pm1$,
\[
\frac{T_n}{\sqrt{n}}\dto {\cal N}_{0, \frac{p}{1-p}}  \ttt{ and\quad} \frac{S_n}{\sqrt{n}}\dto {\cal N}_{0, \frac{p}{1-p}}\ttt{as}\nifi.
\]
\end{theorem}
Our next task is to apply the method of moments in order to prove that this is indeed true. We thus wish to prove that
\bea\label{nfmomente}
E\Big(\frac{T_n}{n}\Big)^{m}\to\mu_{m}=\begin{cases}\dfrac{m!}{2^{m/2}(m/2)!}\cdot\Big(\dfrac{p}{1-p}\Big)^{m/2},&\ttt{if $m$ is even,}\\[2mm] 
0,&\ttt{if $m$ is odd.}\end{cases}
\eea
This amounts to lengthy computations of various  higher order mixed moments. The reason for this is that higher order moments of $T_n$ can be expressed as linear combinations of lower order moments of $T_n$ and $X_n$ with the aid of the binomial theorem.

Convergence of mean and variance has already been established above. For  higher order moments we use induction.

Throughout in the following,  $C(p,m)$, with our without an index,  are numerical constants which may differ from line to line and $E(R_n(p,m))$ are quantities of smaller order than the leading term.
\begin{lemma}\label{momlemma}  For $m\geq 1$ we have, as $n \to \infty$,
\bea
E\big((T_n)^{2m-1}X_n\big)&=&(2p-1)E\big((T_{n-1})^{2m-1}\xnm\big)+(2m-1)E(T_{n-1})^{2m-2}+E(R_n(p,m))
\nonumber\\
&\sim& \frac{2m-1}{2(1-p)}\cdot E(T_{n})^{2m-2}\,;\label{87}\\[1mm]
E\big((T_n)^{2m}X_n\big)&=&(2p-1)E\big((T_{n-1})^{2m}\xnm\big)+2m\cdot E(T_{n-1})^{2m-1}+ \nonumber\\
&&\hskip3pc+ m(2m-1)\,(2p-1) E((T_{n-1})^{2m-2}X_{n-1}) + E(R_n(p,m))\nonumber\\
&\sim& C_1(p,m)\cdot n^{m-1}\,;\label{88}\\
E(\tnp)^{2m}&=&E((T_{n})^{2m})+2m(2p-1)\,\,E((T_n)^{2m-1}X_n) +\nonumber\\
&&\hskip3pc+ m(2m-1)E((T_n)^{2m-2})+\,R_n(p,m)\nonumber\\
&=&E((T_n)^{2m})+ \frac{m(2m-1) \,(2p-1)}{1-p} E((T_n)^{2m-2})+E(R_n(p,m))\nonumber\\
&\sim&\frac{(2m)!}{2^m m!}\Big(\frac{p}{1-p}\Big)^m(n+1)^m\,;\label{89}\\
E(\tnp)^{2m+1}&=&E((T_n)^{2m+1})+ (2m+1)(2p-1)E((T_n)^{2m}X_n)\nonumber\\
&&\hskip3pc+ (2m+1)(m-1) E((T_n)^{2m-1}) + E( R_n(p,m))\nonumber\\
&\sim& C_2(p,m)\cdot (n+1)^m\,,\label{810}
\eea
where $E(R_n(p,m))$ denotes individual remainder terms.
\end{lemma}
The proof of the lemma amounts to extending the above computations for mean and variance to higher order variants and is deferred to the Appendix, Subsection \ref{pfmomlemma}.\smallskip

\pf{Theorem \ref{clttn}} As already mentioned, the proof exploits the method of moments. For $X_1=1$ the lemma tells us that
\[E\big((T_n/\sqrt{n})^{2m}\big) \to \frac{(2m)!}{2^m m!}\Big(\frac{p}{1-p}\Big)^m \ttt{and that} E\big((T_n/\sqrt{n})^{2m+1}\big) \to 0 \ttt{as}\nifi,\]
which verifies  (\ref{nfmomente}).  For $X_1=-1$ we recall from the end of Section \ref{aux} that even moments remain the same and that odd moments are the same except for a change of sign, which yields the same conclusion.
The limit result for $S_n$ then follows as in Theorem \ref{thmgerw}.\vsb
\begin{remark}\emph{The sequence $\{X_n,\,n\geq1\}$ is a stationary recurrent Markov chain with finite state space which, hence, is uniformly ergodic. The asymptotic normality of $T_n$  therefore also follows from a CLT for Markov chains, see, e.g.,  Corollary 5 of \cite{jones} (cf.\ also \cite{ibralinn}, Theorem 19.1.)}\vsb
\end{remark}
The Markov property also provides a strong law.
\begin{theorem}\label{thm82}
We have
\[\frac{S_n}{n} \asto 0 \ttt{as} n \to \infty.\]
\end{theorem}
\begin{proof} The  stationary distribution of the ergodic Markov chain $\{X_n,\,n\geq1\}$ is $(1/2,1/2)$, which has expectation  zero. An application of Theorem 6.1 in \cite{Doob} yields the conclusion.\vsb
\end{proof}

\section{Remembering only the recent past 2}\label{ende2}
\setcounter{equation}{0}
In this section we assume that the elephant remembers the  two most recent steps, that is, at time $n$ the next step is based on the steps $X_{n}$ and $X_{n-1}$. The computations are as before, although  more elaborate. We have, as always, $X_1=1$, $E(X_2\mid {\cal F}_1)=(2p-1)X_1$,  
\[E(X_3\mid {\cal F}_2)=(2p-1)\frac{X_1+X_2}{2}=(2p-1)\frac{1+X_2}{2},\]
and, for $n\geq3$, 
\[E(\xnp\mid \cfn)= E(\xnp\mid \xnm, X_n)=(2p-1)\frac{\xnm+X_n}{2}\]
Computing the moments one obtains the following result. For the proof we refer to the Appendix, Subsection \ref{pflemma91}.
\begin{lemma}\label{lasttwo} As $n\to \infty$,
\beaa
E(X_n) &\to& 0; \\ E(S_n)&\to & \frac{(2p-1)(2p+1)}{4(1-p)};\\ 
\var(S_n/\sqrt{n}) &\to&1+ \frac{(2p-1)(5-2p)}{2(1-p)(3-2p)}=\sigma_2^2\,.
\eeaa
The expectation of $X_n$ tends to zero geometrically fast.
\end{lemma}
\begin{remark}\emph{For $p=1/2$ the process reduces, as usual, to a simple symmetric random walk.}\vsb
\end{remark}
For the following limit theorems we lean on the Markov property (and invite the reader to try the moment method).
\begin{theorem} We have
\[\frac{S_n}{n}\asto 0 \ttt{ and }\frac{S_n}{\sigma_2 \sqrt{n}} \dto {\cal N}_{0,1}\ttt{as}\nifi .\]
\end{theorem}
\begin{proof}
The sequence $\{X_n,\,n\geq1\}$ now forms a Markov chain of order two. Theorem 6.1 in \cite{Doob} yields the strong law, and the results in \cite{herkenrath}, Section 3, or \cite{herkenetal}, combined with Corollary 5 of \cite{jones}, yield the asymptotic normality with the moments as calculated above.\vsb
\end{proof}
\begin{remark} \emph{If we suppose  that the elephant remembers a fixed but finite number, $k$ say, of the most recent steps, the sequence of steps forms a Markov chain of order $k$, and  we obtain, by (basically) the  same arguments as above that $S_n/\sqrt{n}$ will be asymptotically normal
(a Markov chain of order $k$ can be considered as a $k$-dimensional Markov chain and use e.g.\ \cite{herkenetal}).}\vsb
\end{remark}

\section{Remembering the distant as well as the recent past 1}\label{beidesa}
\setcounter{equation}{0}
Next we consider the case when the elephant has a clear memory of the early steps as well as the very recent ones. 

One can think of a(n old) person who remembers the early childhood and events from the last few days but nothing in between. The most elementary case is $\cfn =\sigma\{X_1,X_n\}$, for all $n\geq2$. 
Following the approach of earlier variants we begin by assuming that $X_1=1$. Then, for  $n\geq2$,
\[E(X_2)=E\big(E(X_2\mid X_1)\big)=E\big((2p-1)X_1\big)=2p-1,\]
and
\beaa
E(\xnp)&=&E\big(E(\xnp\mid\cfn)\big)=E\big(E(\xnp\mid X_1,X_n)\big)=(2p-1)E\Big(\frac{X_1+X_n}{2}\Big)\\
&=&(2p-1)E\Big(\frac{1+X_n}{2}\Big)=\frac{2p-1}{2}\cdot(1+E(X_n)).\eeaa
Exploiting Proposition \ref{diff}(i) we obtain, for $n\ge 1$,
\bea \label{expxnn}
E(X_n)=\frac{2p-1}{3-2p}+\Big(\frac{2p-1}{2}\Big)^{n-1}\cdot\frac{4(1-p)}{3-2p}\,,
\eea 
and, hence, that
\bea\label{exptnn}
E(T_n)&=&1+(2p-1) + (n-2)\frac{2p-1}{3-2p}+\frac{4(1-p)}{3-2p}\sum_{k=1}^n \Big(\frac{2p-1}{2}\Big)^{k-1}\nonumber\\
&=& n\cdot\frac{2p-1}{3-2p} +\frac{8(1-p)}{(3-2p)^2} +o(1)\ttt{as}\nifi\,.
\eea
Next we note that $E(T_1^2)=1$, and, by (\ref{uli5}), that, for $n\geq1$,
\bea\label{exptnzprel}
E(\tnpz)=E(T_n^2)+(2p-1)E(T_n) + (2p-1) E(T_nX_n)+1.\eea
In order to establish a difference equation for the second moment we first have to compute the mixed moment. For the computational details we refer to  Appendix \ref{2mom10} and obtain  (formula (\ref{var10})),
\[
E(T_n^2)=\frac{(2p-1)^2}{(3-2p)^2}\cdot n^2+\Big(1+\frac{(2p-1)}{(3-2p)^3}(4p^2-40p+35)\Big)\cdot n +o(n).
\]
Joining the expressions for the first two moments, finally, tells us that the variance is linear in $n$:
\bea\label{vartnzz}
\var(T_n)&=&n^2\cdot\frac{(2p-1)^2}{(3-2p)^2}+n\cdot \Big(1+ \frac{(2p-1)}{(3-2p)^3}(4p^2-40p+35)\Big)\nonumber\\
&&\hskip3pc -\Big( n\cdot\frac{2p-1}{3-2p} +\frac{8(1-p)}{(3-2p)^2}\Big)^2+o(n)\nonumber\\
&=&n\cdot\sigma_T^2 +o(n)\ttt{as}\nifi\,,
\eea
where
\bea\label{sigma10}
\sigma_T^2=1+ \frac{(2p-1)}{(3-2p)^3}(4p^2-24p+19).
\eea
Given the expressions for mean and variance, a weak law is immediate:
\bea\label{wllntn10}
\frac{T_n}{n}\pto \frac{2p-1}{3-2p}\ttt{as}\nifi.
\eea
In analogy with our earlier results this suggets that $T_n$ is asymptotically normal. That this is, indeed, the case follows from the fact that $\{T_n,\ n\geq1\}$ is, once again,  a uniformly ergodic Markov chain, 
since the only random piece from the past is the previous step. We may thus apply Corollary 5 of \cite{jones} (cf.\ also \cite{ibralinn}, Theorem 19.1) to conclude that $T_n-E(T_n)$ is asymptotically normal with mean zero and variance $\sigma^2_Tn$, with $\sigma^2_T$ as defined in (\ref{sigma10}), which, in view of (\ref{exptnn}), establishes  that
\bea\label{clt10prel}
\frac{T_n-\frac{2p-1}{3-2p}n}{\sigma_T\sqrt{n}}\dto {\cal N}_{0,\,1} \ttt{as}\nifi.
\eea
An appeal to the disussion at the end of Section \ref{defs} now allows us to conclude that 
\beaa
 E(S_n)&=&p E(T_n)+(1-p)E(-T_n)=(2p-1)E(T_n),\\
E(S_n^2)&=&E(T_n^2),\\
\var (S_n)&=&E(T_n)^2-(2p-1)^2(E(T_n))^2,
\eeaa
which tells us that 
\bea\label{snlimit}
E(S_n/n)\to\frac{(2p-1)^2}{3-2p}\ttt{and}\var(S_n/n)\to
\sigma_S^2 =4p(1-p)\frac{(2p-1)^2}{(3-2p)^2}\ttt{as}\nifi.
\eea
Furthermore, in analogy to Theorem \ref{thm52}, we arrive at the following asymptotic distributional behavior of $S_n$:
\begin{theorem}\label{snast} We have
\[\frac{S_n}{n}\dto S=\begin{cases}\phantom{-}\dfrac{2p-1}{3-2p},&\ttt{with probability}p,\\[3mm]
 -\dfrac{2p-1}{3-2p},&\ttt{with probability}1-p,\end{cases}\ttt{as}\nifi.\]
Moreover, $E(S_n/n)^r\to E(S^r)$ for all $r>0$, since $|S_n/n|\leq 1$ for all $n$.
\end{theorem}
\begin{remark}\label{rem10}
\emph{Comparing this  with Theorem \ref{thm52} we see that the jump points are closer together here. This can be explained by the fact that the current random variables are less dependent than those in Section 5.}\vsb \end{remark}
Finally, by combining (\ref{clt10prel}) with the obvious analog for the case $X_1=-1$, asymptotic normality follows with a \emph{random} centering:
\begin{theorem}\label{thmclt10} We have
\[\frac{S_n-\frac{(2p-1)X_1}{3-2p}n}{\sigma_T\sqrt{n}}\dto {\cal N}_{0,1} \ttt{as}\nifi.
\]
\end{theorem}
\begin{proof}  We first note that it follows from the discussion following (\ref{clt10prel}) that the CLT there remains true when $X_1=-1$ with a $+$ replacing the $-$ in the numerator. We thus may argue as in the proof of Theorem \ref{thm52}, via the fact that
\[\frac{S_n-\frac{(2p-1)X_1}{3-2p}n}{\sigma_T \sqrt{n}}
=X_1\cdot\frac{T_n-\frac{(2p-1)}{3-2p}n}{\sigma_T \sqrt{n}}.\]
Alternatively, condition on the value of 
$X_1$ and proceed as in the proof of Theorem \ref{thm62}. \vsb
\end{proof}

\section{Remembering the recent as well as the distant past 2}\label{beidesb}
\setcounter{equation}{0}
In this section we extend the previous one in that we assume that $\cfn =\sigma\{X_1,X_2,X_n\}$, for all $n\geq3$. 
Following the approach of earlier variants we begin by assuming that $X_1=X_2=1$. Then $E(X_1)=E(X_2)=1$, $E(X_3)=(2p-1)$  and, for $n\geq3$,
\[E(\xnp)=E\big(E(\xnp\mid\cfn)\big)=E\big(E(\xnp\mid X_1,X_2,X_n)\big)=\frac{2p-1}{3}\cdot(2+E(X_n)).\]
Exploiting Proposition \ref{diff}(i) yields
\bea \label{expxnnn}
E(X_n)=\frac{2p-1}{2-p}+\Big(\frac{2p-1}{3}\Big)^{n-1} \Big(1-\frac{2p-1}{2-p}\Big)
=\frac{2p-1}{2-p}+\frac{3(1-p)}{2-p}\Big(\frac{2p-1}{3}\Big)^{n-2}\,,
\eea 
and, hence,
\bea\label{exptnnn}
E(T_n)&=&1+1+(2p-1)+ (n-3) \cdot\frac{2p-1}{2-p}+\frac{3(1-p)}{2-p}\sum_{k=4}^{n}\Big(\frac{2p-1}{3}\Big)^{k-2}\nonumber\\
&=& n\cdot\frac{2p-1}{2-p} + \frac{3(1-p)(7-2p)}{2(2-p)^2}+o(1)\ttt{as}\nifi\,.
\eea
As for second moments,  $E(T_1)^2=1$, $E(T_2)^2=4$, 
$E(T_3)^2=E(1+1+X_3)^2=4+4E(X_3)+1=4+4(2p-1)+1=8p+1$, and, generally, that,
\bea\label{exptnzprel2}
E(\tnpz)=E(T_n^2)+2E(T_n\xnp)+1.\eea
Concerning the mixed moments and other details we refer to  Appendix \ref{2mom11}, from which  we obtain
\[
E(T_n^2)= n^2\cdot\frac{(2p-1)^2}{(2-p)^2}+n\cdot\Big(1+\frac{(2p-1)}{(2-p)^3}\cdot(5p^2-32p+26)\Big)+o(n).
\]
The variance, finally, turns out as
\bea\label{vartnnn}
\var(T_n)= n\cdot\sigma_T^2+o(n),\ttt{where}\sigma_T^2=1+\frac{(2p-1)}{(2-p)^3}\cdot(5-5p-p^2).
\eea
Following the path of the previous section we now immmediately obtain a weak law:
\bea\label{wllntn11}
\frac{T_n}{n}\pto \frac{2p-1}{2-p}\ttt{as}\nifi.
\eea

 It remains to consider the general case with arbitrary $X_1$ and $X_2$. There is a slight change here from the previous section. Namely, we first have the case when $X_1=X_2=-1$, for which the arguments from the previous section carry over without change, that is, the mean equals $E(-T_n)$ and the second moment equals $E(\tnz)$. However, now we also have a mixed case which behaves somewhat differently.

Namely, consider the case when the first two summands are not equal; $X_1+X_2=0$, $X_1X_2=-1$. Then, 
\[E(X_3)=E\big(E(X_3\mid X_1,X_2)\big)=E\Big((2p-1)\frac{X_1+X_2}{2}\Big)=(2p-1)E(0)=0,\] 
and, for $n\geq3$, 
\beaa
E(\xnp)&=&E\big(E(\xnp\mid\cfn)\big)=E\big(E(\xnp\mid X_1,X_2,X_n)\big)=\frac{2p-1}{3}\cdot\big(0+E(X_n)\big)\\
&=&\frac{2p-1}{3}E(X_n)=\cdots=CE(X_3)=0,\eeaa
from which we conclude that, for $n\geq2$,
\bea\label{exptnnnn}
E(T_n)=0.
\eea
For the calculation of the second moment we refer again to Appendix \ref{2mom11} and find that
\[E(\tnz)=n\cdot\frac{1+p}{2-p}+o(n)=\var (\tnp)\ttt{as}\nifi,\]
where the last equality is due to the fact that $E(T_n)=0$. The weak law now runs slightly differently, in that
\bea\label{wllntn11null}
\frac{T_n}{n}\pto0\ttt{as}\nifi.
\eea
We note in passing that the mean is linear  in $n$ and that  the second moment is of order $n^2$ when the first two summands are equal, whereas  the mean  is zero and the second moment is linear in $n$ when they are not. However, the variance is linear in $n$ in all cases.

As for central limit theorems, the main arguments are the same as in Section \ref{beidesa}, in that
\bea\label{clt11prel}
\frac{T_n\pm\frac{2p-1}{2-p}n}{\sigma_T\sqrt{n}}\dto {\cal N}_{0,1} \ttt{as}\nifi,
\eea
for the cases $X_1=X_2=-1$ and $X_1=X_2=1$, respectively, and
\bea\label{cltprel11null}
\frac{T_n}{\sqrt{n\cdot\frac{1+p}{2-p}}}\dto {\cal N}_{0,1} \ttt{as}\nifi,
\eea
when the first two summands are unequal.

Switching to moments of $S_n$, using $T_n^+$, $T_n^-$ and $T_n^0$ for the three cases, we obtain,   
\beaa
 E(S_n)&=&p^2 E(T_n^+)+ (1-p)\cdot E(T_n^0) + p(1-p)E(T_n^-)=p(2p-1)E(T_n^+),\\[2mm]
E(S_n^2)&=&p^2E((T_n^+)^2)+ (1-p)\cdot E((T_n^0)^2) + p(1-p)E((T_n^-)^2).\\
\eeaa
Collecting the various pieces tells us that
\[
E(S_n/n)\to\frac{p(2p-1)^2}{2-p}\ttt{and}
\var(S_n/n)\to \frac{p(1-p) (2p-1)^2(4p^2+1)}{(2-p)^2}\ttt{as}\nifi.
\]
Finally, by modifying our earlier results of this kind, one ends up as follows: 
\begin{theorem}\label{snast2} We have
\beaa
&&\frac{S_n}{n}\dto S=\begin{cases}\phantom{-}\dfrac{2p-1}{2-p},&\ttt{with probability}p^2,\\[3mm]
\phantom{5555}0,&\ttt{with probability}1-p,\\[3mm]
 -\dfrac{2p-1}{2-p},&\ttt{with probability}p(1-p),\end{cases}\ttt{as}\nifi.
\eeaa
Morevover, $E(S_n/n)^r\to E(S^r)$ for all $r>0$, since $|S_n/n|\leq 1$ for all $n$.
\end{theorem}

We finally wish to combine the asymptotic normality for the three different beginnings of the process in order to arrive at a limit theorem for the $S$-process. This works (in theory) the same way as in Section \ref{beidesa}. However, there is a problem with the variance. Namely, in Theorem \ref{thmclt10} both cases had the same variance, whereas here the variance, when $X_1$ and $X_2$ are equal, is not the same as when they are different. Nevertheless, here is the result.\begin{theorem}\label{thmclt11} We have
\[\frac{S_n-\frac{(2p-1)(X_1+X_2)/2}{2-p}n}{\sqrt{n}}\dto p\cdot{\cal N}_{0,\sigma_T^2}
+(1-p) \cdot{\cal N}_{0,(1+p)/(2-p)} \ttt{as}\nifi,
\]
with $\sigma_T^2$ as given in (\ref{vartnnn}).
\end{theorem}
\begin{proof} The conclusion follows by conditioning on the value of $(X_1+X_2)/2$, and proceeding as in the proof of Theorem \ref{thm62}.\vsb
\end{proof}

\section{Miscellania}\label{anm}
\setcounter{equation}{0}

We close by mentioning some further specific models and by describing some problems and challenges for further research.

\subsection{More on restricted memories}\label{more}
\textbf{(i)} The next logical step would be to check the case when $\cfn=\sigma\{X_1,X_{n-1},X_n\}$. By modifying the computations in Appendix \ref{pflemma91}, setting $a=\frac{2p-1}{3}$ and $d=3a^2$, we find that
\beaa
\mu_{n+1}&=&E(\xnp)=E\big(E(\xnp\mid X_1,\xnm,X_n)\big)=\frac{2p-1}{3}E(X_1+\xnm+X_n)\\ &=&a\big((2p-1)+\mu_{n-1}+\mu_n\big)= a(\mu_{n-1}+\mu_n)+d,\eeaa
after which Proposition \ref{diff}(iv), and a glance at the computations in Appendix \ref{pflemma91}, tell us that
\[E(X_n)= \frac{(2p-1)^2}{5-4 p} + {\cal O }(q^n)\ttt{as}\nifi.\]
where $q=\max\{|\lambda_1|,|\lambda_2|\}<1$, with $\lambda_i$, $i=1,2$, defined in Appendix \ref{pflemma91}, and it follows that
\[ E(S_n)\sim n\,\frac{(2p-1)^2}{5-4p} \ttt{as}\nifi.\]
If $\cfn=\sigma\{X_1,X_2,X_{n-1},X_n\}$, then, with $a=\frac{2p-1}{4}$ and $d=2p(2p-1)^2/4$, one similarly obtains that
\[E(S_n)\sim n\,\frac{p\,(2p-1)^2}{3-2p} \ttt{as}\nifi.\]
In fact, \emph{theoretically\/} it is possible to obtain results of the above kind for any fixed number of early and/or late memory steps.

\smallskip\noindent
\textbf{(ii)} A more subtle case is when the number of memory steps depends on $n$, such as $\log n$ or $\sqrt{n}$.

\smallskip\noindent
\textbf{(iii)} Another model is when the elephant remembers everything \emph{except\/} the first step, more generally,
the elephant remembers all but the first $k$ steps for some $k\in \mathbb{N}$. Set $p_k=P(X_{k+1}=1)$,
$V_n=X_{k+1}+\dots X_n$, and ${\cal H}_n=\sigma\{X_{k+1},\dots X_n\}$, and let $n\ge k+1$. Then 
\[ E(X_{n+1}\mid {\cal H}_n)=(2p-1)\frac{V_n}{n-2} \ttt{ and thus } E(V_{n+1}\mid {\cal H}_n) =\tilde{\gamma}_n V_n,\]
where $\gamma_n= (n-3+2p_k)/(n-2)$. With
\[ \tilde{a}_n=\prod_{k=3}^n \tilde{\gamma}_k^{-1}= \frac{\Gamma(n-1)\, \Gamma(2p_k)}{\Gamma(n-2+2p_k)}\]
one can, as  in \cite{bercu}, show that  $\tilde{a}_n V_n$ is a martingale. From the same paper it follows, provided that $0<p_k<3/4$, that
\[ \frac{V_n}{\sqrt{n-2}} \dto {\cal N}_{0,1/(4-3p_k)}  \ \ttt{ as } n \to \infty,\]
which  implies that
\[ \frac{S_n}{\sqrt{n}} \dto {\cal N}_{0,1/(4-3p_k)}  \ \ttt{ as } n\to \infty.\]
The quantity $p_k$ depends on the construction used for the $k$ steps $X_2,...,X_{k+1}$. \smallskip

Other cases one might think of is when the memory covers everything \emph{except\/}
\begin{itemize}
\item  the last $j$ steps;
\item the first $k$ steps and the last $j$ steps;
\item the first $\alpha\log n$ steps and or the last $\beta\log n$ steps for some $\alpha, \beta>0$;
\item the first $\alpha\sqrt{n}$ steps and or the last $\beta\sqrt{n}$ steps for some $\alpha, \beta>0$;
\item the first $\alpha\log n$ steps and or the last $\beta\sqrt{n}$ steps for some $\alpha, \beta>0$;
\item and so on, aiming at more general (final) results.
\end{itemize}

\subsection{Phase transition}
The results of Bercu \cite{bercu} show that for the full memory one has a phase transition at $p=3/4$. There is no such thing in our results. An obvious, as well as interesting, question would be to find the breaking point. There exist some papers on this topic using simulations, see e.g., \cite{scv,csv, moura} and further  papers cited therein, but we are not aware of any theoretical results concerning this matter.

\subsection{Remembering the first vs.\ the last step}\label{dep}
There is a fundamental difference in behavior in these extreme cases, it is not just a matter of recalling some earlier step. Namely, it is a matter of comparing
\[
\xnp=\begin{cases} +X_{1},\ttt{with probability $p$,} \\-X_{1},\ttt{with probability $1-p$,}\end{cases}\]
with
\[
\xnp=\begin{cases} +X_{n},\ttt{with probability $p$,} \\-X_{n},\ttt{with probability $1-p$.}\end{cases}\]
In order to see the difference more clearly, let us imagine that $p$ is close to one.

In the first case every new step equals most likely the \emph{first\/} one, that is,
a typical path will then constist of an overwhelming amout of $X_1$:s interfoliated by an occasional $-X_1$.
In the second case every new step equals most likely the \emph{most recent\/} one, that is,
a typical path will  constist of an overwhelming amout of $X_1$:s followed by an overwhelming amount of $-X_1$:s, followed by ...., that is alternating long stretches of the same kind.

Moreover, since, in the first case, every new step is a function of just the first one, the independence structure does not come as a surprise, whereas in the second case the next step depends on the previous one, which in turn depends on its previous one, etcetera, which implies that the next step, in fact, depends on the whole path so far.

\subsection{Final Remarks}
\textbf{(i)}  We have seen that the more the elephant remembers the cumbersomer become the computations. However, once again, {\it in theory\/} it would be possible to compute higher order moments and thus, e.g., use the moment method to prove limit theorems. \\[1mm]
\textbf{(ii)} By using the device from Section \ref{allgemein} one can  extend all limit theorems for ERW:s to the case with general steps.\vsb

\renewcommand{\theequation}{A.\arabic{equation}}
\appendix
\section{Appendix}
\label{appa} \setcounter{equation}{0}

In this appendix we collect more technical calculations.

\subsection{Proof of Lemma \ref{momlemma}}\label{pfmomlemma}
Recall that even powers of $X_n$ are always equal to 1, and, moreover, that $X_n^k$=$X_n$ if $k$ is odd. One consequence of this is the following fact that will be used repeatedly below:
\bea\label{blandat3}
E\big((T_n)^{2m-k}(\xnp)^k\mid\cfn\big)&=&E\big((T_n)^{2m-k}E(\xnp\mid\cfn)\big)\nonumber\\
&=&\begin{cases} E(T_n)^{2m-k},&\ttt{if $k$ is even,}\\[2mm] 
E\big((T_n)^{2m-k}(2p-1)X_n\big),&\ttt{if $k$ is odd.}\end{cases}
\eea
 As mentioned in connection with the statement of Theorem \ref{clttn} we use induction. We  thus assume that we know that the moments up to order $2m-2$ converge properly, in particular we may choose $n$ so large that $|E(T_n/\sqrt{n})^k|\leq 2\mu_k$ for $k\leq 2m-2$, which, by symmetry, inplies that $|E(T_n)^k|\leq 4\mu_kn^{k/2}$, for $k\mbox{ even }\leq 2m-1$ and $\leq 4\veps n^{k/2}$, and for $k\mbox{ odd }\leq 2m-1$, for some $\veps$ small (recall that $\mu_k$ are the moments of the standard normal distribution as given in (\ref{nfmomente})).

\noindent
\pf{(\ref{87})}
\beaa&&\hskip-4pc
E\big((T_n)^{2m-1}X_n\mid\cfnm\big)=E\big((\tnm+X_n)^{2m-1}X_n\mid\cfnm\big)\\
&=&\sum_{k=0}^{2m-1}\binom{2m-1}{k}(T_{n-1})^{2m-1-k}\cdot E\big((X_n)^{k+1}\mid\cfnm\big)\\
&=&(\tnm)^{2m-1}E(X_n\mid\cfnm)+(2m-1)(\tnm)^{2m}\cdot E(X_n^2\mid\cfnm)+R_n(p,m)\\
&=&(\tnm)^{2m-1}(2p-1)X_{n-1}+(2m-1)(\tnm)^{2m}+R_n(p,m).\eeaa
Taking expectations on either side yields
\bea\label{87prel}
E\big((T_n)^{2m-1}X_n\big)=(2p-1)E((\tnm)^{2m-1}X_{n-1})+(2m-1)E(\tnm)^{2m-2}+E(R_n(p,m)).
\eea
Exploiting (\ref{blandat3}) yields a bound for the remainder:
\bea
|E(R_n(p,m))|
&\leq&|\sum_{k=3}^{2m-1}\binom{2m-1}{k}E\big((T_{n-1})^{2m-1-k}\cdot E((X_n)^{k+1}\mid\cfnm)\big)|\nonumber\\
&\leq&\sum_{k=3}^{2m-1}\binom{2m-1}{k}\,|E(T_{n-1})^{2m-1-k}|\cdot 1\nonumber\\
&\leq& C_p\cdot 2m\cdot\binom{2m-1}{m-1}\cdot4\mu_{2m-4}\cdot n^{m-2}=C(m,p)\cdot n^{m-2}.
\label{restestimate}
\eea
By iterating (\ref{87prel}) we then obtain that
\beaa
E\big((T_n)^{2m-1}X_n\big)&=&(2m-1)\sum_{k=1}^{n-1} (2p-1)^{n-1-k}E\big((T_k)^{2m-2}\\
&&\hskip3pc +\,R_{k+1}(p,m)/(2m+1)\big)+(2p-1)^{n-1}\\
&=&(2m-1)\frac{1}{2(1-p)}E(T_n)^{2m-2} +{\cal O}\big(n^{m-2}\big).   \eeaa

\noindent
\pf{(\ref{88})}
\beaa&&\hskip-4pc
E\big((T_n)^{2m}X_n\mid\cfnm\big)=E\big((\tnm+X_n)^{2m}X_n\mid\cfnm\big)\\
&=&\sum_{k=0}^{2m}\binom{2m}{k}(T_{n-1})^{2m-k}\cdot E\big((X_n)^{k+1}\mid\cfnm \big)\\
&=&(\tnm)^{2m}E(X_n\mid\cfnm)+2m(\tnm)^{2m-1}\cdot E(X_n^2\mid\cfnm)+R_n(p,m)\\
&=&(\tnm)^{2m}(2p-1)X_{n-1}+2m(\tnm)^{2m-1}+R_n(p,m).\eeaa
Taking expectations on either side yields
\[
E\big((T_n)^{2m}X_n\big)=(2p-1)E\big((\tnm)^{2m}X_{n-1}\big)+2mE(\tnm)^{2m-1}+E(R_n(p,m)).\]
The estimation of the remainder is the same as above. The remaining part of the proof follows the exact same lines and is therefore omitted.

Having estimates for the mixed moments we are now able to attack the ''pure'' moments. This will be done without explicit mentioning. Moreover, the estimates for the remainders, are, again, the same.

\noindent
\pf{(\ref{89})} 
\beaa&&\hskip-4pc
E\big((\tnp)^{2m}\mid\cfn\big)=\sum_{k=0}^{2m}\binom{2m}{k}(T_{n})^{2m-k}\cdot E\big((\xnp)^{k}\mid\cfn\big)\\
&=&(T_n)^{2m}+2m(2p-1)(T_n)^{2m-1}X_n+\binom{2m}{2}(T_n)^{2m-2}+R_n(p,m).\eeaa
Taking expectations on either side yields
\beaa
E(\tnp)^{2m}&=&E(T_n)^{2m}+2m(2p-1)E\big((T_n)^{2m-1}X_n\big)+\binom{2m}{2}E(T_n)^{2m-2}+E(R_n(p,m))\\
&=&E(T_n)^{2m}+2m(2p-1)\cdot(2m-1)\frac{1-(2p-1)^{n-1}}{2(1-p)}E(T_n)^{2m-2}\\&&\hskip3pc +\binom{2m}{2}E(T_n)^{2m-2}+ {\cal O}\big(n^{m-2}\big) \\
&=&E(T_n)^{2m} + \frac{m(2m-1)p}{1-p}\cdot E(T_n)^{2m-2}+ {\cal O}\big(n^{m-1}\big)\\
&=& E(T_n)^{2m} + \frac{m(2m-1)p}{1-p}\frac{(2m-2)!}{2^{m-1} (m-1)!} n^{m-1} \,
+\, {\cal O}\big( n^{m-2}\big),\\
\eeaa
by  the induction hypothesis. Summing up the differences leads to the desired result.\\[2mm]

\noindent
\pf{(\ref{810})} 
\beaa&&\hskip-4pc
E\big((\tnp)^{2m+1}\mid\cfn\big)=\sum_{k=0}^{2m+1}\binom{2m+1}{k}(T_{n})^{2m+1-k}\cdot E\big((\xnp)^{k}\mid\cfn\big)\\
&=&(T_n)^{2m+1}+(2m+1)(2p-1)(T_n)^{2m}X_n+\binom{2m+1}{2}(T_n)^{2m-1}+R_n(p,m).\eeaa
Taking expectations on either side yields
\beaa
E(\tnp)^{2m+1}&=&E(T_n)^{2m+1}+(2m+1)(2p-1)E\big((T_n)^{2m}X_n\big)\\&&\hskip3pc+\binom{2m+1}{2} E(T_n)^{2m-1}+E(R_n(p,m))\\
&=& E(T_n)^{2m+1}+(2m+1)(2p-1)C_1(p,m)\,n^{m-1}\\
&& \hspace*{2cm} +m(2m-1)C_2(p,m-1)\,n^{m-1}+\, {\cal O}\big( n^{m-2}\big)
\eeaa
by the induction hypothesis. 
Summing up, finally,  leads to the desired result with $C_2(p,m)
=\frac{(2m+1)(2p-1)}{m}\cdot C_1(p,m)+(2m-1)\cdot C_2(p,m-1)$.
\vsb

\subsection{Proof of Lemma \ref{lasttwo}}\label{pflemma91}  
Set $a=p-1/2 \in (-1/2,\,1/2)$. Then,
\[ E(X_1)=2a \ttt{and}  E(X_2)= E(X_2 \mid X_1)=2a E(X_1) =4 a^2.\]
For $n\ge 2$ we have
\[ \mu_{n+1} =E(X_{n+1})=E(E(X_{n+1}\mid {\cal F}_n))= E((2p-1)\,(X_{n-1}+X_n)/2)=a\,(\mu_{n-1}+\mu_n).\]
With $\lambda_{1/2}=(a\pm\sqrt{a^2+4a})/2$ (note that $|\lambda_{1/2}|<1$) this difference equation, with the two starting values $2a$ and $4a^2$, has, for $n\ge 1$, the solution
\[ \mu_n=E(X_n)=\frac{a(3a+\sqrt{a^2+4a})}{\sqrt{a^2+4a}} \lambda_1^{n-1}-\frac{a(3a-\sqrt{a^2+4a})}{\sqrt{a^2+4a}}\lambda_2^{n-1}.\]
For $p<1/2$ we have $\sqrt{a^2+4a}=i\sqrt{|a^2+4a|}$, but the solution is still real. Next,
\begin{eqnarray*}
E(S_n)&=&\sum_{k=1}^n \mu_k=\frac{a(3a+\sqrt{a^2+4a})}{\sqrt{a^2+4a}}\frac{1-\lambda_1^n}{1-(a+\sqrt{a^2+4a})/2}\\&&\hskip3pc-\frac{a(3a-\sqrt{a^2+4a})}{\sqrt{a^2+4a}}\frac{1-\lambda_2^n}{1-(a-\sqrt{a^2+4a})/2}\\
&\to& \frac{2a(3a+\sqrt{a^2+4a})}{\sqrt{a^2+4a}(2-a-\sqrt{a^2+4a})}- \frac{2a(3a-\sqrt{a^2+4a})}{\sqrt{a^2+4a}(2-a+\sqrt{a^2+4a})}\\
&=& \frac{2a(a+1)}{1-2a}= \frac{p^2-1/4}{1-p}\ttt{as}\nifi.
\end{eqnarray*}
The second moment is more tedious. We begin with
\[E(S^2_{n+1}\mid{\cal F}_n)=S_n^2 +(2p-1)S_n\, (X_{n-1}+X_n)+1,\]
and obtain
\begin{equation}\label{vn} v_{n+1}= E(S^2_{n+1})=v_n +(2p-1)E(S_nX_n+S_n X_{n-1})+1\,.\end{equation}
As for the mixed moments,
\[ E(S_n X_n\mid {\cal F}_{n-1})=a S_{n-1}X_{n-1} +a(S_{n-2}X_{n-2} +X_{n-1}X_{n-2})+1\,.\]
By the usual trick we find
\[ E(X_nX_{n-1})=a+a^2+\dots+a^{n-1}E(X_2X_1)\to \frac{a}{1-a}= \frac{2p-1}{3-2p}\ttt{as}\nifi.\]
With $\zeta_n=E(S_nX_n)$ we find that
\[\zeta_n=a(\mu_{n-1}+\mu_{n-2})+1+\frac{4a^2}{1-a} + {\cal O}(a^n),\]
from which it follows that $\zeta_n \to \frac{p^2-2p+7/4}{(1-p)(3-2p)}$, the stationary solution. 

 Next,
\[ E(S_nX_{n-1})= E(S_{n-1}X_{n-1})+E(X_nX_{n-1})=\zeta_{n-1}+\frac{a}{1-a} +{\cal O}(a^n).\]
We finally arrive, recalling (\ref{vn}), at
\[v_{n+1}=v_n+(2p-1)\left(\frac{2p^2-4p+7/2}{(1-p)(3-2p)}+\frac{2p-1}{3-2p}\right)+1+ o(1), \]
and thus, via telescoping, at
\[v_n\sim n \Big( 1+\frac{(2p-1)(5-2p)}{2(1-p)(3-2p)}\Big)\,.\]
\vsp

\subsection{Calculation of second moments in Section \ref{beidesa}}\label{2mom10} 
We first note that $E(T_1^2)=1$, and, by (\ref{uli5}), that, for $n\geq1$,
\bea\label{exptnzprel}
E(\tnpz)=E(T_n^2)+(2p-1)E(T_n) + (2p-1) E(T_nX_n)+1.\eea
At this point we have to pause and compute the mixed moments: We first note that $E(T_1X_1)=1$, and that
\[E(X_2X_1)=E\big(X_1 E(X_2\mid X_1)\big)=E(X_1(2p-1)X_1)=2p-1,\]
so that
\[E(T_2X_2\mid\cfn)=E(X_1X_2 +1)=2p-1+1  =2p.\]
For $n\geq 2$ we exploit (\ref{uli4}), (\ref{exptnn}), and the fact that $X_n^2=1$, to obtain
\beaa
E(\tnp\xnp)&=&\frac{2p-1}{2}\cdot E(T_nX_n) + \frac{2p-1}{2} E(T_n) +1\\
&=&\frac{2p-1}{2}\cdot E(T_nX_n)  + \frac{2p-1}{2}\Big( n\cdot\frac{2p-1}{3-2p} +\frac{8(1-p)}{(3-2p)^2} 
+o(1)\Big)+1\\
&=&\frac{2p-1}{2}\cdot E(T_nX_n) +\frac{(2p-1)^2}{2(3-2p)}\cdot n +\frac{4(2p-1)(1-p)}{(3-2p)^2} +1+o(1) .
\eeaa
Another application of Proposition \ref{diff}(i) then tells us that
\bea\label{blandat}
E(T_nX_n)&=&\bigg\{\frac{(2p-1)^2}{2(3-2p)}\cdot(n-1) + \frac{4(2p-1)(1-p)}{(3-2p)^2} +1\Big)\Big/
\Big(1-\frac{2p-1}{2}\Big) 
\nonumber\\
&&\hskip3pc -\frac{2p-1}{2}\cdot\frac{(2p-1)^2}{2(3-2p)}\cdot(n-1)\Big/\Big(\Big(1-\frac{2p-1}{2}\Big)^2
\bigg\}\cdot\big(1+o(1)\big) 
\nonumber\\
&=&\frac{(2p-1)^2}{(3-2p)^2}\cdot n +\frac{8(1+p-2p^2)}{(3-2p)^3}+o(1).
\eea
Hence, using (\ref{exptnzprel}), we obtain
\begin{eqnarray*}
E(T_{n+1}^2)&=& E(T_n)^2+(2p-1) \Big(n\cdot\frac{2p-1}{3-2p} +\frac{8(1-p)}{(3-2p)^2}\Big)\\
&&\hskip3pc + (2p-1)\Big(\frac{(2p-1)^2}{(3-2p)^2}\cdot n +\frac{8(1+p-2p^2)}{(3-2p)^3}\Big)+1+o(1)\\
&=&  E(T_n)^2+\frac{2(2p-1)^2}{(3-2p)^2}\cdot n + \frac{32(2p-1)(1-p)}{(3-2p)^3} +1+o(1),
\end{eqnarray*}
after which we, via telescoping, obtain that   
\begin{eqnarray}\label{var10}
E(T_n^2)&=&  \frac{(2p-1)^2}{(3-2p)^2}\cdot n^2-\frac{(2p-1)^2}{(3-2p)^2}\cdot n
+(n-1)\cdot \Big(1+ \frac{32(2p-1)(1-p)}{(3-2p)^3}\Big)+o(n). \nonumber\\
&=& \frac{(2p-1)^2}{(3-2p)^2}\cdot n^2+\Big(1+\frac{(2p-1)}{(3-2p)^3}(4p^2-40p+35)\Big)\cdot n +o(n).
\end{eqnarray}

\subsection{Calculation of second moments in Section \ref{beidesb}}\label{2mom11} 
The point of departure in this case is (\ref{exptnzprel2}), viz.,
\bea\label{zweimomapp}
E(\tnpz)=E(T_n^2)+2E(T_n\xnp)+1.
\eea
For the mixed moments we use (\ref{uli4}):
\[
E(T_{n}\xnp\mid \cgn)= \frac{2}{3} (2p-1) T_n +\frac{2p-1}{3} T_nX_n 
=  \frac{2}{3} (2p-1) T_n + \frac{2p-1}{3} T_{n-1} X_n +\frac{2p-1}{3}. \]
We thus find, using (\ref{exptnnn}), that for $n\ge 3$,  
\beaa
E(T_n\xnp)&=& \frac{2p-1}{3} E(T_{n-1}X_n) +\frac{2p-1}{3}+ \frac{2(2p-1)}{3}\Big(  n\cdot\frac{2p-1}{2-p} + \frac{3(1-p)(7-2p)}{2(2-p)^2}+o(1)\Big)\\
&=&\frac{2p-1}{3} E(T_{n-1}X_n)+\frac{2(2p-1)^2}{3(2-p)^2}\cdot n + \frac{2p-1)(1-p)(7-2p)}{(2-p)^2}
+\frac{2p-1}{3}+o(1).
\eeaa
Invoking Proposition \ref{diff}(i) then tells us that
\beaa
E(T_n\xnp)&=&\bigg\{\frac{2(2p-1)^2}{3(2-p)^2}\cdot n  + \frac{2p-1)(1-p)(7-2p)}{(2-p)^2}
+\frac{2p-1}{3}\Big) \Big/\Big(1-\frac{2p-1}{3}\Big) \\
&&\hskip3pc - \frac{2p-1}{3}\cdot\frac{2(2p-1)^2}{3(2-p)^2}\cdot n\Big/\Big((n+1)\Big(1-\frac{2p-1}{3}\Big)^2\bigg\}\cdot\big(1+o(1)\big)\\
&=&\frac{(2p-1)^2}{(2-p)^2}\cdot n +\frac{3(2p-1)}{(2-p)^3}\cdot(p^2-9p+8)+o(1)\,,
\eeaa
which, inserted into (\ref{zweimomapp}), yields
\[E(T_{n+1}^2)= E(T_n^2) +\frac{2(2p-1)^2}{(2-p)^2}\cdot n +\frac{3(2p-1)(p^2-9p+8)}{(2-p)^3}+1+ o(1),
\]
and, after summation,
\begin{eqnarray}\label{var11}
E(T_n^2)&=&n^2\cdot\frac{(2p-1)^2}{(2-p)^2}-n\cdot \frac{(2p-1)^2}{(2-p)^2}
+(n-1)\cdot\Big(1+\frac{3(2p-1)(p^2-9p+8)}{(2-p)^3}\Big) +o(n) \nonumber \\
&=& n^2\cdot\frac{(2p-1)^2}{(2-p)^2}+n\cdot\Big(1+\frac{(2p-1)}{(2-p)^3}\cdot(5p^2-32p+26)\Big)+o(n).
\end{eqnarray}

We, finally, turn our attention to the second moment for the case when $X_1\cdot X_2=-1$, where, again,
the mixed moment is first in focus. Now, $E(T_1X_1)=1$, $E(T_2X_2)=E(X_1X_2+X_2^2)=-1+1=0$, and
$E(T_3X_3)=E(T_2X_3+X_3)^2=0+1=1$. For $n\geq3$ we follow the usual pattern. Due to the fact that the mean is zero, an application of  (\ref{uli4}) now yields
\bea\label{etnpxnpnull}
E(T_n\xnp)= \frac{2p-1}{3} E(T_nX_n) = \frac{2p-1}{3} E(T_{n-1}X_n)+\frac{2p-1}{3},
\eea
which, together with Proposition \ref{diff}(i), tells us that
\bea\label{etnxnpp}
E(T_n\xnp)=\frac{\frac{2p-1}{3}}{1-\frac{2p-1}{3}}+o(1)=\frac{2p-1}{2(2-p) }+o(1)\ttt{as}\nifi.
\eea
Moving into second moments, $E(T_1^2)=1$, $E(T_2^2)=0$, and
$E(T_3^2)=E(X_3^2)=1$. For $n\geq3$ we insert our findings in (\ref{etnxnpp}) into (\ref{zweimomapp}):
\beaa
E(\tnpz)&=&E(T_n^2)+2E(T_n\xnp)+1=E(T_n)^2+\frac{2p-1}{2-p}+1+o(1)\\
&=&E(T_n^2)+\frac{1+p}{2-p}+o(1)\ttt{as}\nifi,
\eeaa
so that, via telescoping,
\bea\label{exptnzzzzz}
E(\tnz)=n\cdot\frac{1+p}{2-p}+o(n)=\var (\tnp)\ttt{as}\nifi.
\eea

\subsection*{Acknowledgement}

The results of this paper were initiated during U.S.'s visit in Uppsala in May 2018. U.S.\ wants to thank for the kind hospitality and we both wish to thank  Kungliga Vetenskapssamh\"allet i Uppsala for financial support.

\medskip\noindent {\small Allan Gut, Department of Mathematics,
Uppsala University, Box 480, SE-751\,06 Uppsala, Sweden;\\
Email:\quad \texttt{allan.gut@math.uu.se}\\
URL:\quad \texttt{http://www.math.uu.se/\~{}allan}}
\\[4pt]
{\small Ulrich Stadtm\"uller, Ulm University, Department of Number
Theory
and Probability Theory,\\ D-89069 Ulm, Germany;\\
Email:\quad \texttt{ulrich.stadtmueller@uni-ulm.de}\\
URL:\quad
\texttt{http://www.mathematik.uni-ulm.de/en/mawi/institute-of-number-theory-and-probability-\\theory/people/stadtmueller.html}}

\end{document}